\documentclass[10pt]{amsart}
\usepackage{amssymb, amsthm, amsmath}
\usepackage[all]{xy}
\usepackage{graphicx}
\usepackage{psfrag}

\newtheorem{prop}{Proposition}
\newtheorem{thm}{Theorem}

\theoremstyle{definition}

\newtheorem{defn}{Definition}
\newtheorem{example}{Example}
\newtheorem{remark}{Remark}

\newcommand\A{{\mathbb A}}

\newcommand\N{{\mathbb N}}

\newcommand{\Ti}{\Theta}

\newcommand{\om}{{\varpi}}

\newcommand\Z{{\mathbb Z}}

\newcommand\bW{{\mathbb W}}

\newcommand\AS{{\mathfrak S}}
\newcommand\BS{{\mathfrak B}}
\newcommand\CS{{\mathfrak C}}
\newcommand\DS{{\mathfrak D}}

\newcommand\la{\lambda}

\newcommand\Eta{H}

\newcommand\ssm{\smallsetminus}

\newcommand\eqto{\stackrel{\lower1.5pt\hbox{$\scriptstyle\sim\,$}}\to}
\newcommand\ov{\overline}

\newcommand\wt{\widetilde}
\newcommand\dis{\displaystyle}

\DeclareMathOperator{\type}{\mathrm{type}}

\newcommand{\ignore}[1]{}

\begin{document}

\title[Tableau formulas for skew Schubert polynomials]
{Tableau formulas for skew Schubert polynomials}

\date{March 21, 2023}

\author{Harry~Tamvakis} \address{University of Maryland, Department of
Mathematics, William E. Kirwan Hall, 4176 Campus Drive, 
College Park, MD 20742, USA}
\email{harryt@umd.edu}


\begin{abstract}
The skew Schubert polynomials are those which are indexed by skew
elements of the Weyl group, in the sense of \cite{T1}. We obtain
tableau formulas for the double versions of these polynomials in all
four classical Lie types, where the tableaux used are fillings of the
associated skew Young diagram. These are the first such theorems for
symplectic and orthogonal Schubert polynomials, even in the single
case. We also deduce tableau formulas for double Schur, double theta,
and double eta polynomials, in their specializations as double
Grassmannian Schubert polynomials. The latter results generalize the
tableau formulas for symmetric (and single) Schubert polynomials due
to Littlewood (in type A) and the author (in types B, C, and D).
\end{abstract}

\maketitle

\setcounter{section}{-1}

\section{Introduction}

The double Schubert polynomials of Lascoux-Sch\"utzenberger \cite{LS,
  L2} and Ikeda-Mihalcea-Naruse \cite{IMN} represent the (stable)
Schubert classes in the equivariant cohomology ring of complete flag
manifolds, in each of the four classical Lie types. When the indexing
Weyl group element is {\em skew}, in the sense of \cite{T1, T3}, we
call these polynomials -- which are really formal power series in
types B, C, and D -- {\em skew Schubert polynomials}.  The aim here is
to prove tableau formulas for the skew Schubert polynomials, in a type
uniform manner. By definition, each skew signed permutation is
associated with a pair of (typed) partitions $\la\supset \mu$, and our
formulas are sums over tableaux which are fillings of the boxes in the
skew diagram $\la/\mu$.

The skew elements of the symmetric group coincide with the
$321$-avoiding or fully commutative permutations \cite{BJS, S}. Their
Schubert polynomials were named and studied by Lascoux and
Chen-Yan-Yang \cite{CYY}, following the work of Wachs \cite{W} and
Billey-Jockusch-Stanley \cite{BJS} in the single case. The results
here are new even in type A, and provide an alternative to the
formulas in \cite{CYY}, which extends readily to the symplectic and
orthogonal Lie types.

Our theorems specialize to give the first tableau formulas for the
single Schubert polynomials of Billey-Haiman \cite{BH} indexed by skew
signed permutations.  The well known single skew Schur $S$-, $P$-, and
$Q$-functions are not skew Schubert polynomials, and it has been
a longstanding open problem to formulate a theory of such polynomials
in types B, C, and D, even in the fully commutative case. We remark
that the (type A and single) skew Schubert polynomials of
Lenart-Sottile \cite{LeS} are different from the ones found in
\cite{BJS, CYY} and the present paper.

The double {\em Grassmannian Schubert polynomials} are the images of
the double Schur, double theta, and double eta polynomials of
\cite{KL, L1, TW, T5} in the ring of double Schubert polynomials
of \cite{LS, IMN} of the corresponding Lie type. Since the
Grassmannian elements are the most important examples of skew
elements, our results also specialize to obtain formulas for
Grassmannian Schubert polynomials. These expressions in turn
generalize the tableau formulas for single Schur, theta, and eta
polynomials found in \cite{Li} and \cite{T1, T3}, respectively. The
latter objects are the {\em symmetric Schubert polynomials}, where the
symmetry means invariance under the action of the respective Weyl
group, as explained in \cite{T6}.

The simple proof -- which is new, and uniform across the four types --
stems from the raising operator approach to tableau formulas and the
ensuing theory of skew elements of the Weyl group, pioneered in
\cite{T1, T3}.  We also employ the definition of double Schubert
polynomials via the nilCoxeter algebra found in \cite{T2}, which
originates in the work of Fomin-Stanley-Kirillov-Lam \cite{FS, FK,
  La}.  These ingredients combine in a harmonious way to yield
straightforward arguments. In a sequel to this paper, we illustrate
the power of these methods further by extending our results to {\em
skew Grothendieck polynomials}.

This article is organized as follows. Section \ref{wgsp} contains
preliminary material on the relevant Weyl groups and double Schubert
polynomials. The following Sections \ref{tAt}, \ref{tCt}, and
\ref{tDt} deal in a parallel manner with tableau formulas for skew
Schubert polynomials in the Lie types A, C, and D, respectively.

I thank the anonymous referee for a careful reading of the paper and
detailed suggestions which helped to improve the exposition.

\section{Weyl groups and Schubert polynomials}
\label{wgsp}

This section gathers background information on the double Schubert
polynomials for the classical Lie groups due to Lascoux and
Sch\"utzenberger \cite{LS, L2} (in type A) and Ikeda, Mihalcea, and
Naruse \cite{IMN} (in types B, C, and D). We require the definition of
these formal power series using the nilCoxeter algebra of the Weyl
group, which originates in \cite{FS, FK, La}, and was used in
\cite{T1, T2, T3}. The precise way in which the Schubert polynomials
studied here represent the stable equivariant Schubert classes on
complete flag manifolds is explained in \cite{IMN}; see also \cite{T2, T4}.

The Weyl group for the root system of type $\text{B}_n$ or
$\text{C}_n$ is the {\em hyperoctahedral group} $W_n$, which consists
of signed permutations on the set $\{1,\ldots,n\}$. The group $W_n$ is
generated by the transpositions $s_i=(i,i+1)$ for $1\leq i \leq n-1$
and the sign change $s_0(1)=\ov{1}$ (as is customary, we set
$\ov{a}:=-a$ for any $a\geq 1$). The elements of $W_n$ are written as
$n$-tuples $(w_1,\ldots, w_n)$, where $w_i:=w(i)$ for each $i\in
[1,n]$.

There is a natural embedding $W_n\hookrightarrow W_{n+1}$ defined by
adding the fixed point $n+1$, and we let $W_\infty :=\cup_n W_n$. The
{\em length} of an element $w\in W_\infty$, denoted $\ell(w)$, is the
least integer $r$ such that we have an expression $w=s_{a_1} \cdots
s_{a_r}$. The word $a_1\cdots a_r$ is called a {\em reduced word} for
$w$.  The symmetric group $S_n$ is the subgroup of $W_n$ generated by
$s_1,\ldots, s_{n-1}$, and we let $S_\infty:=\cup_n S_n$.

The nilCoxeter algebra $\bW_n$ of $W_n$ is the free unital associative
algebra generated by the elements $\xi_0,\xi_1,\ldots,\xi_{n-1}$
modulo the relations
\[
\begin{array}{rclr}
\xi_i^2 & = & 0 & i\geq 0\ ; \\
\xi_i\xi_j & = & \xi_j\xi_i & |i-j|\geq 2\ ; \\
\xi_i\xi_{i+1}\xi_i & = & \xi_{i+1}\xi_i\xi_{i+1} & i>0\ ; \\
\xi_0\xi_1\xi_0\xi_1 & = & \xi_1\xi_0\xi_1\xi_0.
\end{array}
\]
For every $w\in W_n$, define $\xi_w := \xi_{a_1}\ldots \xi_{a_r}$, where
$a_1\cdots a_r$ is any reduced word for $w$. The elements $\xi_w$ for
$w\in W_n$ form a free $\Z$-basis of $\bW_n$. We denote the
coefficient of $\xi_w\in \bW_n$ in the expansion of the element
$\alpha\in \bW_n$ by $\langle \alpha,w\rangle$.

Let $t$ be an indeterminate and define
\begin{gather*}
A_i(t) := (1+t \xi_{n-1})(1+t \xi_{n-2})\cdots (1+t \xi_i) \ ; \\ 
\tilde{A}_i(t) := (1-t \xi_i)(1-t \xi_{i+1})\cdots (1-t \xi_{n-1}) \ ; \\ 
C(t) := (1+t \xi_{n-1})\cdots(1+t\xi_1)(1+t\xi_0)
(1+t\xi_0)(1+t \xi_1)\cdots (1+t \xi_{n-1}).
\end{gather*}
Suppose that $X=(x_1,x_2,\ldots)$, $Y=(y_1,y_2,\ldots)$, and
$Z=(z_1,z_2,\ldots)$ are three infinite sequences of commuting
independent variables.  For any $\om\in S_n$, the type A Schubert
polynomial $\AS_\om$ is given by
\begin{equation}
\label{dbleA}
\AS_\om(X,Y) := \left\langle 
\tilde{A}_{n-1}(y_{n-1})\cdots \tilde{A}_1(y_1)
A_1(x_1)\cdots A_{n-1}(x_{n-1}), \om\right\rangle.
\end{equation}
Let $C(Z):=C(z_1)C(z_2)\cdots$, and for $w\in W_n$, define the type C Schubert 
polynomial $\CS_w$ -- which is a formal power series in the $Z$ variables -- by
\begin{equation}
\label{dbleC}
\CS_w(Z;X,Y) := \left\langle 
\tilde{A}_{n-1}(y_{n-1})\cdots \tilde{A}_1(y_1)C(Z) A_1(x_1)\cdots 
A_{n-1}(x_{n-1}), w\right\rangle.
\end{equation}
The type C Stanley function $F_w$ of \cite{BH, FK, La} is
given by $F_w(Z):=\left\langle C(Z), w\right\rangle$.  The polynomial
$\CS_w$ is stable under the inclusion of $W_n$ in $W_{n+1}$; it
follows that $\CS_w$ and $\AS_\om$ are well defined for $w\in
W_\infty$ and $\om\in S_\infty$, respectively.

For any $w\in W_\infty$, the type B Schubert polynomial $\BS_w$
satisfies $\BS_w=2^{-s(w)}\CS_w$, where $s(w)$ denotes the number of
indices $i$ such that $w_i<0$. We therefore omit any further
discussion of type B, and will focus on the even orthogonal type D.

The Weyl group $\wt{W}_n$ for the root system $\text{D}_n$ is the
subgroup of $W_n$ consisting of all signed permutations with an even
number of sign changes.  The group $\wt{W}_n$ is an extension of $S_n$
by $s_\Box:=s_0s_1s_0$, an element which acts on the right by
\[
(w_1,w_2,\ldots,w_n)s_\Box=(\ov{w}_2,\ov{w}_1,w_3,\ldots,w_n).
\]
There is a natural embedding $\wt{W}_n\hookrightarrow \wt{W}_{n+1}$ of
Weyl groups defined by adjoining the fixed point $n+1$, and we let
$\wt{W}_\infty := \cup_n \wt{W}_n$. The simple reflections in
$\wt{W}_\infty$ are indexed by the members of the set $\N_\Box
:=\{\Box,1,2,\ldots\}$, and are used to define the length and reduced
words of elements in $\wt{W}_\infty$ as above.

The nilCoxeter algebra $\wt{\bW}_n$ of the group $\wt{W}_n$ is the free unital
associative algebra generated by the elements
$\xi_\Box,\xi_1,\ldots,\xi_{n-1}$ modulo the relations
\[
\begin{array}{rclr}
\xi_i^2 & = & 0 & i\in \N_\Box\ ; \\
\xi_\Box \xi_1 & = & \xi_1 \xi_\Box \\
\xi_\Box \xi_2 \xi_\Box & = & \xi_2 \xi_\Box \xi_2 \\
\xi_i\xi_{i+1}\xi_i & = & \xi_{i+1}\xi_i\xi_{i+1} & i>0\ ; \\
\xi_i\xi_j & = & \xi_j\xi_i & j> i+1, \ \text{and} \ (i,j) \neq (\Box,2).
\end{array}
\]

For any element $w\in \wt{W}_n$, choose a reduced word $a_1\cdots a_r$
for $w$, and define $\xi_w := \xi_{a_1}\ldots \xi_{a_r}$.  As before,
denote the coefficient of $\xi_w\in \wt{\bW}_n$ in the expansion of
the element $\alpha\in \wt{\bW}_n$ in the $\xi_w$ basis by $\langle
\alpha,w\rangle$. Following Lam \cite{La}, define
\[
D(t) := (1+t \xi_{n-1})\cdots (1+t \xi_2)(1+t \xi_1)(1+t \xi_\Box)
(1+t \xi_2)\cdots (1+t \xi_{n-1}).
\]
Let $D(Z):=D(z_1)D(z_2)\cdots$, and for $w\in \wt{W}_n$, define
the type D Schubert polynomial $\DS_w$ by 
\begin{equation}
\label{dbleD}
\DS_w(Z;X,Y) := \left\langle 
\tilde{A}_{n-1}(y_{n-1})\cdots \tilde{A}_1(y_1) D(Z) A_1(x_1)\cdots 
A_{n-1}(x_{n-1}), w\right\rangle.
\end{equation}
The type D Stanley function $E_w$ of \cite{BH, La} is defined by
$E_w(Z):=\left\langle D(Z), w\right\rangle$.  The Schubert polynomial
$\DS_w(Z;X,Y)$ is stable under the natural inclusions
$\wt{W}_n\hookrightarrow \wt{W}_{n+1}$, and hence is well defined for
$w\in \wt{W}_\infty$.

Given any Weyl group elements $u_1,\ldots,u_r,w$, we say that the
product $u_1\cdots u_r$ is a {\em reduced factorization} of $w$ if
$u_1\cdots u_r=w$ and $\ell(u_1)+\cdots + \ell(u_r)=\ell(w)$.

\section{Tableau formula for type A skew Schubert polynomials}
\label{tAt}

\subsection{Grassmannian permutations and partitions}
\label{gpps}

We recall here some standard definitions and notation. A {\em
  partition} $\la=(\la_1,\la_2,\ldots)$ is a weakly decreasing
sequence of nonnegative integers with finite support. The length of
$\la$ is the number of non-zero parts $\la_i$. We identify a partition
$\la$ with its Young diagram of boxes, which are arranged in left
justified rows, with $\la_i$ boxes in the $i$-th row for each $i\geq
1$.  An inclusion $\mu\subset\la$ of partitions corresponds to the
containment relation of their respective diagrams; in this case, the
skew diagram $\la/\mu$ is the set-theoretic difference $\la\ssm\mu$. A
skew diagram is called a {\em horizontal strip} (respectively, {\em vertical
  strip}) if it does not contain two boxes in the same column
(respectively, row).

Fix an integer $m\geq 1$. An element $\om\in S_\infty$ is
$m$-Grassmannian if $\ell(\om s_i)>\ell(\om)$ for all $i\neq m$. This is 
equivalent to the conditions 
\[
\om_1<\cdots < \om_m \quad \mathrm{and} \quad  \om_{m+1}<\om_{m+2}<\cdots.
\]
Every $m$-Grassmannian permutation $\om\in S_\infty$ corresponds to a unique
partition $\la$ of length at most $m$, called the {\em shape}
of $\om$. When the shape $\la$ and $m$ are given, we denote $\om=\om(\la,m)$ by $\om_\la$,
and have $\om_\la(i) = \la_{m+1-i}+i$ for $1\leq i \leq m$.

\subsection{Skew permutations and main theorem}

A permutation $\om\in S_\infty$ is called {\em skew} if there exists
an $m$-Grassmannian permutation $\om_\la$ (for some $m$) and a reduced
factorization $\om_\la = \om\om'$ in $S_\infty$. In this case, the
right factor $\om'$ equals $\om_\mu$ for some $m$-Grassmannian
permutation $\om_\mu$, and we have $\mu\subset\la$.  We say that
$(\la,\mu)$ is a {\em compatible pair} and that $\om$ is associated to
the pair $(\la,\mu)$. There is a 1-1 correspondence between reduced
factorizations $uv$ of $\om$ and partitions $\nu$ with
$\mu\subset\nu\subset\la$ (this is a special case of
\cite[Cor.\ 8]{T1}).  It follows from \cite[Thm.\ 4.2]{S} that the
skew permutations coincide with the fully commutative elements of
$S_\infty$.

\begin{defn}
\label{ddiudef}
We say that a permutation $\om$ is {\em decreasing down to $p$} if
$\om$ has a reduced word $a_1\cdots a_r$ such that $a_1 > \cdots > a_r
\geq p$. We say that $\om$ is {\em increasing up from $p$} if $\om$
has a reduced word $a_1\cdots a_r$ such that $p\leq a_1 < \cdots < a_r$.
\end{defn}

If a permutation $\om$ has a reduced word $a_1\cdots a_r$ such that
$a_1 > \cdots > a_r$, then $a_1$ is the largest integer $i$ such that
$\ell(s_i\om)<\ell(\om)$. It follows by induction on $\ell(\om)$ that
if $\om$ is decreasing down to $p$ or increasing up from $p$, then the
decreasing (respectively, increasing) word $a_1\cdots a_r$ for $\om$ in
Definition \ref{ddiudef} is uniquely determined.

Observe that an $m$-Grassmannian element $s_i\om\in S_n$ satisfies
$\ell(s_i\om)>\ell(\om)$ if and only if $\om=(\cdots i \cdots | \cdots
i+1 \cdots)$, where the vertical line $|$ lies between $\om_m$ and
$\om_{m+1}$. Using this and the relation between $\la$ and $\om_\la$
explained above, one sees that a skew permutation $\om=\om_\la\om_\mu^{-1}$ is decreasing
down to $1$ (respectively, increasing up from $1$) if and only
if $\la/\mu$ is a horizontal (respectively, vertical) strip.

Let $\la$ and $\mu$ be any two partitions of length at most $m$ with
$\mu\subset\la$, and choose $n\geq 1$ such that $\om_\la\in S_n$. Let {\bf P}
denote the ordered alphabet 
\[
(n-1)'<\cdots<1'<1<\cdots<n-1. 
\] 
The symbols $(n-1)',\ldots,1'$ are said to be {\em marked}, while the rest
are {\em unmarked}.

\begin{defn}
An {\em $m$-bitableau} $U$ of shape $\la/\mu$ is a filling of the boxes in $\la/\mu$
with elements of {\bf P} which is weakly increasing along each row and
down each column, such that (i) the marked (respectively, unmarked) entries are strictly
increasing each down each column (respectively, along each row), and (ii) the entries in 
row $i$ lie in the interval $[(\mu_i+m+1-i)',\la_i+m-i]$ for each $i\in [1,m]$.
We define
\[
(xy)^U:= \prod_ix_i^{n'_i}\prod_i (-y_i)^{n_i} 
\]
where $n'_i$ (respectively, $n_i$) denotes the number of times that $i'$
(respectively, $i$) appears in $U$.
\end{defn}

\begin{thm}
\label{Askew} For the skew permutation $\om:=\om_{\la}\om_{\mu}^{-1}$, we have
\begin{equation}
\label{ASteq}
\AS_\om(X,Y)=\sum_U (xy)^U
\end{equation}
summed over all $m$-bitableaux $U$ of shape $\la/\mu$.
\end{thm}
\begin{proof}
It follows from formula (\ref{dbleA}) and the remark after Definition \ref{ddiudef} that
\begin{equation}
\label{Asum}
\AS_\om(X,Y) = \sum_{v_{n-1}\cdots v_1 u_1\cdots u_{n-1}=\om} (-y_{n-1})^{\ell(v_{n-1})}\cdots (-y_1)^{\ell(v_1)}
x_1^{\ell(u_1)}\cdots x_{n-1}^{\ell(u_{n-1})}
\end{equation}
where the sum is over all reduced factorizations $v_{n-1}\cdots v_1
u_1\cdots u_{n-1}$ of $\om$ such that $v_p$ is increasing up from $p$
and $u_p$ is decreasing down to $p$ for each $p\in [1,n-1]$. Since the
elements $u_p$ and $v_p$ involved are all skew permutations, such
factorizations correspond to sequences of partitions
\[
\mu = \la^{0} \subset \la^1 \subset \cdots \subset \la^{n-1}
\subset \la^n \subset \cdots \subset \la^{2n-2} =\la
\]
with $\la^i/\la^{i-1}$ a horizontal strip for $1\leq i\leq n-1$ and a
vertical strip for $n\leq i\leq 2n-2$, defined by the equations
$\om_{\la^i}=u_{n-i}\om_{\la^{i-1}}$, for $i\in [1,n-1]$, and
$\om_{\la^i}=v_{i+1-n}\om_{\la^{i-1}}$, for $i\in [n,2n-2]$. Note
that some factors $u_p$ or $v_p$ in the product $v_{n-1}\cdots v_1u_1\cdots
u_{n-1}$ may be trivial, and in this case, the associated skew
diagram $\la^i/\la^{i-1}$ is empty.  We obtain a corresponding filling
$U$ of the boxes in $\la/\mu$ by placing the entry $(n-i)'$ in each
box of $\la^i/\la^{i-1}$ for $1\leq i\leq n-1$ and the entry $i+1-n$
in each box of $\la^i/\la^{i-1}$ for $n\leq i\leq 2n-2$.

Consider the left action of the reflections $s_a$, for $a$ in the
reduced word of $v_{n-1}\cdots v_1 u_1\cdots u_{n-1}$, on the
$m$-Grassmannian permutations going from $\om_\mu$ to $\om_\la$.  Choose
$j\in [1,m]$ and set $e:=\om_\mu(m+1-j)$ and $f:=\om_\la(m+1-j)$. The leftmost
entry $h_1$ of $U$ in row $j$ of $\la/\mu$ was added by the reflection
$s_e$, hence we must have $h_1\geq e'$. Similarly, the rightmost entry
$h_2$ in row $j$ was added by $s_{f-1}$, therefore we must have
$h_2\leq f-1$. Since the entries of $U$ are clearly weakly increasing
along row $j$, they all lie in the interval $[e',f-1]$. It follows
that $U$ is an $m$-bitableau of shape $\la/\mu$ such that
$(xy)^U=(-y_{n-1})^{\ell(v_{n-1})}\cdots
(-y_1)^{\ell(v_1)}x_1^{\ell(u_1)}\cdots
x_{n-1}^{\ell(u_{n-1})}$. Conversely, the $m$-bitableaux $U$ of shape
$\la/\mu$ correspond to reduced factorizations of $\om$ as in
(\ref{Asum}).  Since the sum in equation (\ref{ASteq}) is over all
such $U$, the result follows.
\end{proof}

\begin{example}
For any $r\geq 1$, we let $X_r:=(x_1,\ldots,x_r)$ and
$Y_r:=(y_1,\ldots,y_r)$.  Let $\om$ be an $m$-Grassmannian permutation
and $\la$ be the corresponding partition. The Schubert polynomial
$\AS_{\om}(X,Y)$ is equal to the {\em double Schur polynomial}
$s_\la(X_m,Y)$, while $s_\la(X_m):=s_\la(X_m,0)$ is the corresponding
single Schur polynomial.  Equation (\ref{ASteq}) in this case reads
\begin{equation}
\label{Slaeq}
s_\la(X_m,Y) = \sum_U (xy)^U
\end{equation}
summed over all fillings $U$ of the boxes in $\la$ with elements of
{\bf P} which are weakly increasing along each row and down each
column, such that the marked (respectively, unmarked) entries are strictly
increasing each down each column (respectively, along each row), and the
entries in row $i$ lie in the interval $[m',\la_i+m-i]$ for $1\leq
i\leq m$.  The reader may compare (\ref{Slaeq}) with the similar result
in \cite[Prop.\ 4.1]{M}.

Equation (\ref{Slaeq}) implies the known formula (see e.g.\ \cite[Prop.\ 4.1]{K})
\[
s_\la(X_m,Y) = \sum_{\mu\subset \la}
s_\mu(X_m)\det\left(e_{\la_i-\mu_j-i+j}(-Y_{\la_i+m-i})\right)_{1\leq i,j\leq m},
\]
where $e_p(-Y_r)$ denotes the $p$-th elementary symmetric polynomial
in $-y_1,\ldots,-y_r$.  Indeed, the marked entries in each
$m$-bitableau $U$ on $\la$ form a filling of a diagram $\mu$ contained
in $\la$. The map $i'\mapsto m+1-i$ shows that these fillings are in
bijection with semistandard Young tableaux $T$ of shape $\mu$ with
entries in $[1,m]$. For each fixed partition $\mu\subset \la$, the
corresponding monomials $x^T$ sum to give $s_\mu(X_m)$, a polynomial
which is {\em symmetric} in the variables $X_m$. The rest follows from
the determinantal formula for flagged skew Schur functions given in
\cite[Thm.\ $3.5^*$]{W}.
\end{example}

\section{Tableau formula for type C skew Schubert polynomials}
\label{tCt}

\subsection{Grassmannian elements and $k$-strict partitions}
\label{GskC}

The main references for this subsection are \cite{BKT1, BKT2, T1}.
Fix a nonnegative integer $k$.  An element $w\in W_\infty$ is
$k$-Grassmannian if $\ell(ws_i)>\ell(w)$ for all $i\neq k$. This is
equivalent to the conditions
\[
0<w_1<\cdots < w_k \quad \mathrm{and} \quad  w_{k+1}<w_{k+2}<\cdots.
\]
A partition $\la$ is said to be {\em $k$-strict} if no part greater
than $k$ is repeated. The number of parts $\la_i$ which are greater
than $k$ is denoted by $\ell_k(\la)$.

Each $k$-Grassmannian element $w$ of $W_\infty$ corresponds to a
unique $k$-strict partition $\la$, called the {\em shape} of $w$. If
the shape $\la$ is given, then we denote the corresponding element
$w=w(\la,k)$ by $w_\la$.  To describe this bijection, let
$\gamma_1\leq\cdots \leq \gamma_k$ be the lengths of the first $k$
columns of $\la$, listed in increasing order. The sequence
$(\gamma_1,\ldots,\gamma_k)$ is the {\em A-code} of $w_\la$, following
\cite[Def.\ 2]{T6}. We then have
\[
w_\la(j)=\gamma_j+j-\#\{p\in [1,\ell_k(\la)]\, :\, \la_p+p > \gamma_j+j+k\}
\]
for $1\leq j \leq k$, while the equalities $w_\la(k+i)=k-\la_i$ for
$1\leq i \leq \ell_k(\la)$ specify the negative entries of
$w_\la$. For example, the $3$-strict partition $\lambda = (8,4,2,1)$
satisfies $\ell_k(\la)=2$ and $(\gamma_1,\gamma_2,\gamma_3)=(2,3,4)$,
therefore $w_\la=(2,4,7,\ov{5},\ov{1},3,6)$.  The correspondence
between $w$ and its shape $\la$ can be visualized using the notion of
related and non-related diagonals; see \cite[Sec.\ 6.1]{BKT2}.

 We denote the box in row $r$ and column $c$ of a Young diagram by
 $[r,c]$.  For any partition $\la$, we define $\la_0:=\infty$ and
 agree that the diagram of $\la$ includes all boxes $[0,c]$ in row
 zero.  The {\em rim} of $\la$ is the set of boxes $[r,c]$ of its
 Young diagram such that box $[r+1,c+1]$ lies outside of the diagram
 of $\la$. We say that the boxes $[r,c]$ and $[r',c']$ are {\em
   $k'$-related} if $|c-k-\frac{1}{2}|+r = |c'-k-\frac{1}{2}|+r'$. For
 instance, the two grey boxes in the figure below are $k'$-related.
 We call the box $[r,c]$ a {\em left box} if $c \leq k$ and a {\em
   right box} if $c>k$.

\[
\includegraphics[scale=0.60]{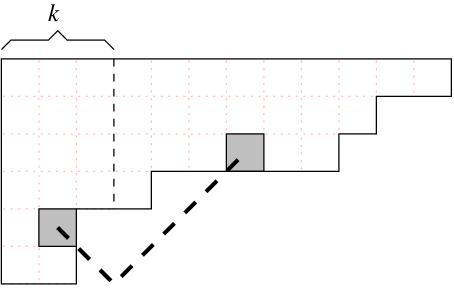}
\]

Following \cite[Sec.\ 5.3]{T1}, if $\mu\subset\la$ are two $k$-strict
partitions, we let $R$ (respectively, $\A$) denote the set of right
boxes of $\mu$ (including boxes in row zero) which are bottom boxes of
$\la$ in their column and are (respectively, are not) $k'$-related to a
left box of $\la/\mu$. The pair $\mu\subset\la$ forms a {\em
  $k$-horizontal strip} $\la/\mu$ if (i) $\la/\mu$ is contained in the
rim of $\la$, and the right boxes of $\la/\mu$ form a horizontal
strip; (ii) no two boxes in $R$ are $k'$-related; and (iii) if two
boxes of $\la/\mu$ lie in the same column, then they are $k'$-related
to exactly two boxes of $R$, which both lie in the same row. 

Note that a $k$-horizontal strip $\la/\mu$ is a pair of partitions
$(\la,\mu)$, and so, depends on $\la$ and $\mu$ and not only on the
difference $\la\ssm\mu$.  We say that two boxes in $\A$ are connected
if they share a vertex or an edge, and let $n(\la/\mu)$ denote the
number of connected components of $\A$ which do not have a box in
column $k+1$.

\subsection{Skew elements and main theorem}
\label{ktmtC}

Following \cite[Sec.\ 6.3]{T1}, an element $w\in W_\infty$ is called
{\em skew} if there exists a $k$-Grassmannian element $w_\la$ (for
some $k$) and a reduced factorization $w_\la = ww'$ in $W_\infty$. In
this case, the right factor $w'$ equals $w_\mu$ for some $k$-strict
partition $\mu$ with $\mu\subset\la$. We say that $(\la,\mu)$ is a
{\em compatible pair} and that $w$ is associated to the pair
$(\la,\mu)$. According to \cite[Cor.\ 8]{T1}, there is a bijection
between reduced factorizations $uv$ of $w_\la w_\mu^{-1}$ and
$k$-strict partitions $\nu$ with $\mu\subset\nu\subset\la$ such that
$(\la,\nu)$ and $(\nu,\mu)$ are compatible pairs. Moreover, any
$k$-horizontal strip $\la/\mu$ is a compatible pair $(\la,\mu)$ of
$k$-strict partitions.

\begin{remark}
The integer $k$, the compatible pair $(\la,\mu)$, and the skew shape
$\la/\mu$ associated to a skew element $w\in W_\infty$ are not
uniquely determined by $w$. For example, the $2$-Grassmannian element
$s_1s_2\in W_3$ is also a skew element when $k=1$, and in the latter
capacity is associated to both of the compatible pairs $((4,1),3)$ and
$((4,3),(3,2))$.
\end{remark}

An element of $W_\infty$ is called {\em unimodal} if it has a
reduced word $a_1\cdots a_r$ such that for some $q\in [0,r]$, we have
$a_1>a_2>\cdots>a_q<a_{q+1}<\cdots<a_r$.  Let $\la$ and $\mu$ be any
two $k$-strict partitions such that $(\la,\mu)$ is a compatible pair,
choose an integer $n\geq 1$ such that $w_\la\in W_n$, and let
$w:=w_\la w_\mu^{-1}$ be the corresponding skew element of $W_n$. It
was shown in \cite[Prop.\ 5]{T1} that $w$ is unimodal if and only if
$\la/\mu$ is a $k$-horizontal strip; we note that this can also be 
checked directly.

\begin{defn}
\label{xydef}
Suppose that $w=w_\la w_\mu^{-1}$ lies in $S_n$. If $w$ is decreasing
down to $1$ (respectively, increasing up from $1$), then we say that the
$k$-horizontal strip $\la/\mu$ is an {\em $x$-strip} (respectively, {\em
  $y$-strip}).
\end{defn}

The $x$- and $y$-strips are characterized among all $k$-horizontal strips as follows.

\begin{prop}
\label{khoriz}
A $k$-horizontal strip $\la/\mu$ is an $x$-strip (respectively, $y$-strip) if
and only if $\ell_k(\la)=\ell_k(\mu)$, the left boxes in $\la/\mu$
form a vertical strip (respectively, horizontal strip), and no two boxes in
$\la/\mu$ are $k'$-related (respectively, no two right boxes in $\la/\mu$ are
in the same row).
\end{prop}
\begin{proof}
Let $w:=w_\la w_\mu^{-1}$ be the skew element of $W_n$ associated to
$\la/\mu$.  Then clearly $w\in S_n$ if and only if $s(w_\la)=s(w_\mu)$
if and only if $\ell_k(\la)=\ell_k(\mu)$. An element of $S_n$ is
decreasing down to $1$ (respectively, increasing up from $1$) if and only if
it has no reduced word which contains $i-1, i$ (respectively, $i, i-1$) as a
subword for some $i\geq 2$. Notice that in any reduced factorization
$w=us_{i-1}s_iv$ (respectively, $w=us_is_{i-1}v$), the elements $u$ and $v$
are also skew, and associated to $k$-horizontal strips which are
substrips of $\la/\mu$. Therefore, by induction on the lengths of $u$
and $v$, we may assume that $w\in \{s_{i-1}s_i, s_is_{i-1}\}$ for some
$i$ and study the associated skew diagram $\la/\mu$.  For $i\geq 1$, a
$k$-Grassmannian element $s_iv\in W_n$ satisfies $\ell(s_iv)>\ell(v)$
if and only if $v$ has one of the following three forms:
\[
(\cdots i \cdots | \cdots i+1 \cdots), \quad (\cdots i+1 \cdots |
\cdots \ov{i} \cdots), \quad (\cdots | \cdots \ov{i} \cdots i+1
\cdots)
\]
where the vertical line $|$ lies between $v_k$ and $v_{k+1}$. We
deduce that if $w=s_{i-1}s_i$ for some $i\geq 2$, then $w_\mu$ and
$w_\la$ have the form
\[
w_\mu = (\cdots i-1,\, i \cdots | \cdots i+1 \cdots),  \ \, 
w_\la = (\cdots i, \, i+1 \cdots | \cdots i-1 \cdots)
\]
so that $\la/\mu$ has two left boxes in the same row, or 
\[
w_\mu = (\cdots i+1 \cdots | \cdots \ov{i},\, \ov{i-1} \cdots), \ \, 
w_\la = (\cdots i-1 \cdots | \cdots \ov{i+1},\, \ov{i} \cdots)
\]
so that $\la/\mu$ has two right boxes which are $k'$-related, or 
\[
w_\mu = (\cdots i \cdots | \cdots \ov{i-1}, \, i+1 \cdots), \ \, 
w_\la = (\cdots i+1 \cdots | \cdots \ov{i},\, i-1 \cdots)
\]
in which case $\la/\mu$ has a left box and a right box which are
$k'$-related. On the other hand, if $w=s_is_{i-1}$, then we must have
\[
w_\mu = (\cdots i-1 \cdots | \cdots i, \, i+1 \cdots), \ \,
w_\la = (\cdots i+1 \cdots | \cdots i-1, \, i \cdots)
\]
so that $\la/\mu$ has two left boxes in the same column, or 
\[
w_\mu = (\cdots | \cdots \ov{i-1} \cdots), \ \,
w_\la = (\cdots | \cdots \ov{i+1} \cdots)
\]
so $\la/\mu$ has two right boxes in the same row. Finally, the converse assertions 
are proved by using the correspondence between $\nu$ and $w_\nu$ given in Section
\ref{GskC}. 
\end{proof}

A {\em $k$-tableau} $T$ of shape $\la/\mu$ is a sequence of
 $k$-strict partitions
\[
\mu = \la^0\subset\la^1\subset\cdots\subset\la^p=\la
\]
such that $\la^i/\la^{i-1}$ is a $k$-horizontal strip for $1\leq i\leq
p$.  We represent $T$ by a filling of the boxes in $\la/\mu$ with
positive integers such that for each $i$, the boxes in $T$ with entry
$i$ form the skew diagram $\la^i/\la^{i-1}$. For any $k$-tableau $T$
we define $n(T):=\sum_i n(\la^i/\la^{i-1})$ and set $z^T:=\prod_i
z_i^{m_i}$, where $m_i$ denotes the number of times that $i$ appears
in $T$.  According to \cite[Thm.\ 6]{T1}, the type C Stanley function $F_w(Z)$
of the skew signed permutation $w=w_{\la}w_{\mu}^{-1}$ satisfies the equation
\begin{equation}
\label{Ftabeq}
F_w(Z) = \sum_T 2^{n(T)} z^T
\end{equation}
summed over all $k$-tableaux $T$ of shape $\la/\mu$.

Let {\bf Q} denote the ordered alphabet
\[
(n-1)'<\cdots 2'<1'< 1 < 2 < 3 < \cdots <1''<2''<\cdots<(n-1)''. 
\]
The single and double primed symbols in {\bf Q} are said to be {\em
  marked}, while the rest are {\em unmarked}.

\begin{defn}
\label{Ctrit}
A {\em $k$-tritableau} $U$ of shape $\la/\mu$ is a filling of the
boxes in $\la/\mu$ with elements of {\bf Q} which is weakly increasing
along each row and down each column, such that (i) for each $a$ in
{\bf Q}, the boxes in $\la/\mu$ with entry $a$ form a $k$-horizontal
strip, which is an $x$-strip (respectively, $y$-strip) if $a\in [(n-1)', 1']$
(respectively, $a\in [1'',(n-1)'']$), and (ii) for $1\leq i \leq \ell_k(\mu)$
(respectively, $1\leq i \leq \ell_k(\la)$) and $1\leq j \leq k$, the entries
of $U$ in row $i$ are $\geq (\mu_i-k)'$ (respectively, $\leq (\la_i-k-1)''$)
and the entries in column $k+1-j$ lie in the interval $[(w_\mu(j))',
  (w_\la(j)-1)'']$.  We define
\[
n(U):=n(T) \quad \mathrm{and} \quad
(xyz)^U:= z^T\prod_ix_i^{n'_i}\prod_i (-y_i)^{n''_i} 
\]
where $T$ is the $k$-tableau formed by the unmarked entries in $U$, and 
$n'_i$ and $n''_i$ denote the number of times that $i'$ and $i''$
appear in $U$, respectively.
\end{defn}

\begin{thm}
\label{Cskew} 
For the skew element $w:=w_{\la}w_{\mu}^{-1}$, we have
\begin{equation}
\label{CSteq}
\CS_w(Z;X,Y)=\sum_U 2^{n(U)}(xyz)^U
\end{equation}
summed over all $k$-tritableaux $U$ of shape $\la/\mu$.
\end{thm}
\begin{proof}
It is clear from formula (\ref{dbleC}) that
\begin{equation}
\label{CSw}
\CS_w(Z;X,Y) 
= \sum (-y_{n-1})^{\ell(v_{n-1})}\cdots (-y_1)^{\ell(v_1)}
F_\sigma(Z) x_1^{\ell(u_1)}\cdots x_{n-1}^{\ell(u_{n-1})}
\end{equation}
where the sum is over all reduced factorizations $v_{n-1}\cdots v_1
\sigma u_1\cdots u_{n-1}$ of $w$ such that $v_p\in S_n$ is increasing
up from $p$ and $u_p\in S_n$ is decreasing down to $p$ for each $p$.
Such factorizations correspond to sequences of $k$-strict partitions
\[
\mu = \la^{0} \subset \la^1 \subset \cdots \subset \la^{n-1}
\subset \la^n \subset \la^{n+1} \subset \cdots \subset \la^{2n-1} =\la
\]
with $\la^i/\la^{i-1}$ an $x$-strip for each $i\leq n-1$,
$(\la^n,\la^{n-1})$ a compatible pair, and $\la^i/\la^{i-1}$ a
$y$-strip for each $i\geq n+1$. The partitions $\la^i$ are determined
by the equations $w_{\la^i}=u_{n-i}w_{\la^{i-1}}$, for $i\in [1,n-1]$,
$w_{\la^n}=\sigma w_{\la^{n-1}}$, and
$w_{\la^i}=v_{i-n}w_{\la^{i-1}}$, for $i\in [n+1,2n-1]$. Note that
some factors in the product $v_{n-1}\cdots v_1\sigma u_1\cdots
u_{n-1}$ may be trivial, and in this case, the corresponding skew
diagram $\la^i/\la^{i-1}$ is empty. We extend each $k$-tableau $T$ on
$\la^n/\la^{n-1}$ to a filling $U$ of the boxes in $\la/\mu$ by
placing the entry $(n-i)'$ in each box of $\la^i/\la^{i-1}$ for $1\leq
i\leq n-1$ and $(i-n)''$ in each box of $\la^i/\la^{i-1}$ for $n+1\leq
i\leq 2n-1$.

Consider the left action of the reflections $s_a$ for $a$ in the
reduced word of the product $v_{n-1}\cdots v_1 \sigma u_1\cdots
u_{n-1}$ on the $k$-Grassmannian signed permutations going from
$w_\mu$ to $w_\la$.  The left action of $s_a$ on a $k$-Grassmannian
element $w_\nu$ must be of the form $(\cdots a \cdots | \cdots a+1
\cdots)\mapsto (\cdots a+1 \cdots | \cdots a \cdots)$ or $(\cdots
\ov{a} \cdots)\mapsto (\cdots \ov{a+1} \cdots)$; these add a left box
or a right box to $\nu$, respectively. It follows as in the proof of
Theorem \ref{Askew} that for $1\leq i \leq \ell_k(\mu)$ (respectively, $1\leq
i \leq \ell_k(\la)$), the entries of $U$ in row $i$ are $\geq
(\mu_i-k)'$ (respectively, $\leq (\la_i-k-1)'')$. Observe that when a right
box is added to $\nu$, at most one of the values
$w_\nu(1),\ldots,w_\nu(k)$ will decrease by $1$.  We deduce that for
$1\leq j \leq k$, the top entry $h_1$ of $U$ in column $k+1-j$ of
$\la/\mu$ was added by a reflection $s_a$ where $a\leq w_\mu(j)$,
hence we must have $h_1\geq (w_\mu(j))'$. Similarly, the bottom entry
$h_2$ of $U$ in column $k+1-j$ must satisfy $h_2\leq
(w_\la(j)-1)''$. Since the entries of $U$ are clearly weakly
increasing down column $k+1-j$, they must all be in the interval
$[(w_\mu(j))',(w_\la(j)-1)'']$. It follows that the entries of $U$ in
the first $k$ columns will lie within the intervals which are listed
in Definition \ref{Ctrit}. We deduce that every filling $U$ of
$\la/\mu$ as above is a $k$-tritableau on $\la/\mu$ such that
$(xyz)^U=(-y_{n-1})^{\ell(v_{n-1})}\cdots
(-y_1)^{\ell(v_1)}z^Tx_1^{\ell(u_1)}\cdots
x_{n-1}^{\ell(u_{n-1})}$. Conversely, the $k$-tritableaux of shape
$\la/\mu$ correspond to reduced factorizations of $w$ of the required
form. Finally, by combining (\ref{CSw}) with (\ref{Ftabeq}), we obtain
(\ref{CSteq}).
\end{proof}

\begin{example} 
Following \cite{TW}, for any $k$-strict partition $\la$, there is a
double theta polynomial $\Ti_\la(c\, |\, t)$, whose image $\Ti_\la(Z;
X,Y)$ in the ring of type C Schubert polynomials is equal to the
Grassmannian Schubert polynomial $\CS_{w_\la}(Z; X,Y)$. Formula
(\ref{CSteq}) therefore gives
\begin{equation}
\label{Tieq}
\Ti_\la(Z;X,Y)=\sum_U 2^{n(U)}(xyz)^U
\end{equation}
summed over all $k$-tritableaux $U$ of shape $\la$. Equation
(\ref{Tieq}) extends \cite[Thm.\ 5]{T1} from single to double theta
polynomials.
\end{example}

\begin{example} 
We extend \cite[Example 7]{T1} to include $k$-tritableaux. Let $k:=1$,
$\la:=(3,1)$, and $Z_2:=(z_1,z_2)$. We have $w_\la=(3,\ov{2},1)$ and
will compute $\Ti_{(3,1)}(Z_2;X,Y)= \CS_{3\ov{2}1}(Z_2;X,Y)$. Consider
the alphabet $\text{\bf Q}_{1,2}=\{2'<1'<1<2<1''<2''\}$. The twelve
$1$-tritableaux of shape $\la$ with entries in $\{2',1',1,2\}$ are
listed in loc.\ cit. There are sixteen further $1$-tritableaux of
shape $\la$ involved. The tritableau $U=\dis \begin{array}{l}
  1\,2\,1'' \\ 2 \end{array}$ satisfies $n(U)=3$, the seven
tritableaux
\[
\begin{array}{l} 1\,1\,1'' \\ 2 \end{array}  \ \
\begin{array}{l} 1\,2\,1'' \\ 1 \end{array}  \ \
\begin{array}{l} 1'\,1\,1'' \\ 1 \end{array}  \ \
\begin{array}{l} 1'\,2\,1'' \\ 1 \end{array}  \ \
\begin{array}{l} 1'\,1\,1'' \\ 2 \end{array}  \ \
\begin{array}{l} 1'\,2\,1'' \\ 2 \end{array}  \ \
\begin{array}{l} 1\,2\,1'' \\ 1'' \end{array}  
\]
satisfy $n(U)=2$, while the eight tritableaux
\[
\begin{array}{l} 1\,1\,1'' \\ 1 \end{array}  \ \
\begin{array}{l} 2\,2\,1'' \\ 2 \end{array}  \ \
\begin{array}{l} 1'\,1\,1'' \\ 1' \end{array} \ \
\begin{array}{l} 1'\,2\,1'' \\ 1' \end{array} \ \
\begin{array}{l} 1\,1\,1'' \\ 1'' \end{array}  \ \
\begin{array}{l} 2\,2\,1'' \\ 1'' \end{array}  \ \
\begin{array}{l} 1'\,1\,1'' \\ 1'' \end{array}  \ \
\begin{array}{l} 1'\,2\,1'' \\ 1'' \end{array}  
\]
satisfy $n(U)=1$. 
We deduce from \cite[Example 7]{T1} and Theorem \ref{Cskew} that
\begin{align*}
\Ti_{3,1}(Z_2; X,Y) 
&= \Ti_{3,1}(Z_2; X) -y_1(2z_1^3+8z_1^2z_2+8z_1z_2^2+2z_2^3) \\
&\qquad -y_1(4z_1^2+8z_1z_2+4z_2^2)x_1-y_1(2z_1+2z_2)x_1^2 \\
& \qquad + y_1^2(2z_1^2+4z_1z_2+2z_2^2) + y_1^2(2z_1+2z_2)x_1 \\
&= \Ti_{3,1}(Z_2; X) - y_1\Ti_{2,1}(Z_2; X) + y_1^2 \Ti_2(Z_2; X).
\end{align*}
\end{example}

\begin{remark}
Suppose that $w$ is a skew element of $W_\infty$ associated to the
compatible pair $(\la,\mu)$ of $k$-strict partitions. One may
view the right hand side of (\ref{CSteq}) as a tableau formula for
a `double skew theta polynomial' $\Ti_{\la/\mu}(Z;X,Y)$ indexed by
$\la/\mu$. However, Wilson's double theta polynomials $\Ti_\la(c\,|\,
t)$ from \cite{TW}, like their single versions $\Ti_\la(c)$ in
\cite{BKT2}, are defined in terms of raising operators acting on
monomials in a different set of variables. It remains an open question
to determine an analogue of these raising operator formulas in the
skew case.  A similar remark applies to the theory of single and
double eta polynomials found in \cite{BKT3, T5}. See also Example
\ref{Qex} below, which examines Theorem \ref{Cskew} when $k=0$.
\end{remark}

\begin{example}
\label{Qex}
Suppose that $k=0$ and let $w:=w_{\la}w_{\mu}^{-1}$ be the fully
commutative skew element associated to a pair $\la\supset\mu$ of
$0$-strict partitions.  According to \cite[Thm.\ 6.6]{IMN}, the Schubert
polynomial $\CS_{w_\la}(Z;X,Y)$ is equal to a double analogue
$Q_\la(Z;Y)$ of Schur's $Q$-function introduced by Ivanov \cite{I},
and we also have $\CS_{w_\mu}(Z;X,Y) = Q_\mu(Z;Y)$. However, the skew
Schubert polynomial $\CS_w$ will in general involve both the $X$ and
$Y$ variables. For instance, assume that $\la =
\delta_n:=(n,n-1,\ldots,1)$, so that $w_\la =
w_{\delta_n}=(\ov{n},\ldots,\ov{1})$ is the longest $0$-Grassmannian
element in $W_n$. Since
$w_\mu^{-1}w^{-1}=w_{\delta_n}^{-1}=w_{\delta_n}$, we see that in this
case $w^{-1}=w_{\mu^\vee}$ is the $0$-Grassmannian element with shape
given by the $0$-strict partition $\mu^\vee$ whose parts complement
the parts $\mu_i$ of $\mu$ in the set $\{1,\ldots,n\}$.  Using the
symmetry property of double Schubert polynomials
\cite[Thm.\ 8.1]{IMN}, we conclude that $\CS_w(Z;X,Y) =
\CS_{w^{-1}}(Z;-Y,-X)= Q_{\mu^\vee}(Z;-X)$. It is an instructive
exercise to deduce this equality from the tableau formula
(\ref{CSteq}).
\end{example}

\section{Tableau formula for type D skew Schubert polynomials}
\label{tDt}

\subsection{Grassmannian elements and typed $k$-strict partitions}
\label{Getks}

The main references for this subsection are \cite{BKT1, BKT3, T3, T7}.
According to \cite[Def.\ 1]{T7}, we say that $w\in \wt{W}_\infty$ has
type $0$ if $|w_1|=1$, type $1$ if $w_1>1$, and type 2 if $w_1<-1$.
Fix a positive integer $k$. An element $w\in \wt{W}_\infty$ is
$k$-Grassmannian if $\ell(ws_i)>\ell(w)$ for all $i\neq k$, if $k>1$,
and for all $i>1$, if $k=1$.  This is equivalent to the conditions
\[
|w_1|<\cdots < w_k \quad \mathrm{and} \quad  w_{k+1}<w_{k+2}<\cdots
\]
with the first condition being vacuous if $k=1$. Following \cite{T7}, we
regard the $\Box$-Grassmannian elements as a subset of the
$1$-Grassmannian elements.

A {\em typed $k$-strict partition} is a pair consisting of a
$k$-strict partition $\la$ together with an integer $\type(\la)\in
\{0,1,2\}$, which is positive if and only if $\la_i=k$ for some index
$i$.  There is a type-preserving bijection between the
$k$-Grassmannian elements $w$ of $\wt{W}_\infty$ and typed $k$-strict
partitions $\la$. If the element $w$ corresponds to the typed
partition $\la$, then we denote $w=w(\la,k)$ by $w_\la$.

Given a typed $k$-strict partition $\la$, the bijection is determined
as follows.  Let $\ell_k(\la)$ and $\gamma_1\leq\cdots \leq \gamma_k$
be defined as in Section \ref{GskC}. If $\type(\la)\neq 2$, then 
\[
w_\la(j)=\gamma_j+j-\#\{p\in [1,\ell_k(\la)] \, :\, \la_p +p \geq \gamma_j+j+k\}
\]
for $1\leq j \leq k$, while the equalities $w_\la(k+i)=k-1-\la_i$ for
$1\leq i \leq \ell_k(\la)$ give the negative entries of $w_\la$ which
are less than $-1$.  If $\type(\la)=2$ and $\la'$ is the partition of
type 1 with the same shape as $\la$, then $w_\la$ is related to
$w_{\la'}=(w'_1,\ldots,w'_n)$ by changing the sign of the first entry
$w'_1$ and of the entry $w'_p$ with $|w'_p|=1$.  For example, the
typed $3$-strict partition $\lambda = (7,4,3,2)$ of type $2$ satisfies
$\ell_k(\la)=2$ and $(\gamma_1,\gamma_2,\gamma_3)=(3,4,4)$, therefore
$w_\la=(\ov{3},6,7,\ov{5},\ov{2},\ov{1},4,8)$. We refer to
\cite[Sec.\ 6.1]{BKT3} for a picture of the correspondence between $w$
and $\la$ which uses related and non-related diagonals.

We say that the boxes $[r,c]$ and $[r',c']$ in a Young diagram are
{\em $(k-1)$-related} if $|c-k|+r = |c'-k|+r'$.  For instance, the two
grey boxes in the figure below are $(k-1)$-related. We call the box
$[r,c]$ a {\em left box} if $c \leq k$ and a {\em right box} if $c>k$.

\[
\includegraphics[scale=0.60]{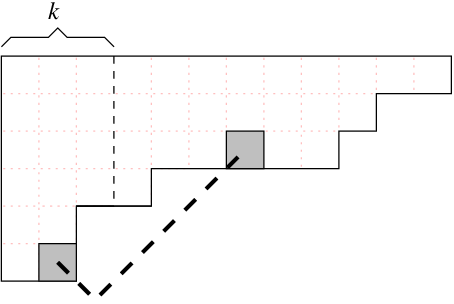}
\]

\medskip

Following \cite[Sec.\ 2.6 and 3.4]{T3}, if $\mu\subset\la$ are two
typed $k$-strict partitions, we let $R$ (respectively, $\A$) denote the set
of right boxes of $\mu$ (including boxes in row zero) which are bottom
boxes of $\la$ in their column and are (respectively, are not)
$(k-1)$-related to a left box of $\la/\mu$.  A pair $\mu\subset\la$ of
typed $k$-strict partitions forms a {\em typed $k'$-horizontal strip}
$\la/\mu$ if $\type(\la)+\type(\mu)\neq 3$ and (i) $\la/\mu$ is
contained in the rim of $\la$, and the right boxes of $\la/\mu$ form a
horizontal strip; (ii) no two boxes in $R$ are $(k-1)$-related; and
(iii) if two boxes of $\la/\mu$ lie in the same column, then they are
$(k-1)$-related to exactly two boxes of $R$, which both lie in the
same row.  We define $n'(\la/\mu)$ to be one less than the number of
connected components of $\A$.

\subsection{Skew elements and main theorem}

Following \cite[Sec.\ 4.3]{T3}, an element $w\in \wt{W}_\infty$ is
called {\em skew} if there exists a $k$-Grassmannian element $w_\la$
(for some $k\geq 1$) and a reduced factorization $w_\la = ww'$ in
$\wt{W}_\infty$. In this case, the right factor $w'$ equals $w_\mu$
for some $k$-Grassmannian element $w_\mu$, and we have
$\mu\subset\la$. We say that $(\la,\mu)$ is a {\em compatible pair}
and that $w$ is associated to the pair $(\la,\mu)$. According to
\cite[Cor.\ 2]{T3}, there is a 1-1 correspondence between reduced
factorizations $uv$ of $w_\la w_\mu^{-1}$ and typed $k$-strict
partitions $\nu$ with $\mu\subset\nu\subset\la$ such that $(\la,\nu)$
and $(\nu,\mu)$ are compatible pairs. Any typed $k'$-horizontal strip
$\la/\mu$ is an example of a compatible pair $(\la,\mu)$ of typed
$k$-strict partitions.

Let $\la$ and $\mu$ be any two typed $k$-strict partitions such that
$(\la,\mu)$ is a compatible pair, choose an integer $n\geq 1$ such
that $w_\la\in \wt{W}_n$, and let $w:=w_\la w_\mu^{-1}$ be the
corresponding skew element of $\wt{W}_n$.  It was shown in \cite{T3}
that $w$ is unimodal (in the sense of Section \ref{ktmtC}) if and only
if $\la/\mu$ is a typed $k'$-horizontal strip.

\begin{defn}
We say that a typed $k'$-horizontal strip $\la/\mu$ is {\em extremal} if
\[
(\ell_k(\la),\type(\la))\neq (\ell_k(\mu),\type(\mu)).
\]  
For any typed $k$-strict partition $\nu$, let
$\epsilon(\nu):=\ell_k(\nu)+\type(\nu)$. 
\end{defn}

\begin{defn} 
\label{xyDdef}
Suppose that $w=w_\la w_\mu^{-1}$ lies in $S_n$. If $w$ is decreasing
down to $1$ (respectively, increasing up from $1$), then we say that the
typed $k'$-horizontal strip $\la/\mu$ is a {\em typed $x$-strip}
(respectively, {\em typed $y$-strip}).
\end{defn}

The typed $x$- and typed $y$-strips are characterized among all typed
$k'$-horizontal strips by the next result.

\begin{prop}
A typed $k'$-horizontal strip $\la/\mu$ is a typed $x$-strip
(respectively, typed $y$-strip) if and only if {\em (i)} the left boxes in
$\la/\mu$ form a vertical strip (respectively, horizontal strip), and no two
boxes in $\la/\mu$ are $(k-1)$-related (respectively, no two right boxes in
$\la/\mu$ are in the same row), and {\em (ii)} if $\la/\mu$ is
extremal then $(\type(\la),\type(\mu))\neq (0,0)$ and the following
condition holds: if $\epsilon(\mu)$ is odd, then $\epsilon(\la)$ is
odd {\em and} $\type(\mu)=0$, while if $\epsilon(\mu)$ is even, then
$\epsilon(\la)$ is odd {\em or} $\type(\mu)=1$.
\end{prop}
\begin{proof}
For $i\geq 1$, a $k$-Grassmannian element $s_iv\in \wt{W}_n$ satisfies
$\ell(s_iv)>\ell(v)$ if and only if $v$ has one of the following four
forms:
\[
(\cdots i \cdots i+1 \cdots),\quad (\cdots i+1 \cdots \ov{i} \cdots), \quad 
(\cdots \ov{i} \cdots i+1 \cdots), \quad (\ov{i+1} \cdots \ov{i} \cdots).
\]
Moreover, a $k$-Grassmannian element $s_\Box v$ satisfies $\ell(s_\Box v)>\ell(v)$
if and only if $v$ has one of the following four forms:
\[
(\cdots 1 \cdots 2 \cdots), \quad
(\ov{1} \cdots 2 \cdots), \quad
(2 \cdots 1 \cdots), \quad
(\ov{2} \cdots 1 \cdots).
\]

Let $w:=w_\la w_\mu^{-1}$ be the skew element of $\wt{W}_n$ associated
to $\la/\mu$. It is easy to verify that $\la/\mu$ is extremal if and
only if $w(1)\neq 1$. In this case, we have
$(\type(\la),\type(\mu))\neq(0,0)$ if and only if no unimodal reduced
word for $w$ has $\Box 1$ as a subword. Now suppose that $\la/\mu$ is
extremal, $w$ lies in $S_n$, and $w$ is decreasing down to $1$ or
increasing up from $1$. A case-by-case analysis shows that if
$\epsilon(\mu)$ is odd, then $\epsilon(\la)$ is odd {\em and}
$\type(\mu)=0$, while if $\epsilon(\mu)$ is even, then $\epsilon(\la)$
is odd {\em or} $\type(\mu)=1$. On the other hand, for the skew
element $s_0ws_0$ of shape $\la/\mu$, we observe that if
$\epsilon(\mu)$ is odd, then $\epsilon(\la)$ is even {\em or}
$\type(\mu)=1$, while if $\epsilon(\mu)$ is even, then $\epsilon(\la)$
is even {\em and} $\type(\mu)=0$. Indeed, since $w(1)\neq 1$, we are
reduced to examining what happens when $w_\la=s_1 w_\mu$ and
$w_\la=s_\Box w_\mu$, respectively. The $10 = 2\cdot 5$ different cases
are illustrated in Table \ref{mula} when $k=2$ and $n=4$, and the
picture for other values of $k$ and $n$ follows the same pattern. The
remainder of the argument is similar to the proof of Proposition
\ref{khoriz}.
\end{proof}

{\small{
\begin{table}[t]
\caption{Statistics for the extremal pairs $u$ and $s_1u$, $v$ and $s_\Box v$}
\centering
\begin{tabular}{|c|c|c|c|c||c|c|c|c|c|} \hline
$u$ & $\mu$ & $\ell_k(\mu)$ & type$(\mu)$ & $\epsilon(\mu)$ 
& $s_1u$ & $\la$ & $\ell_k(\la)$ & type$(\la)$ & $\epsilon(\la)$ \\ \hline
$1324$ & 1 & even & 0 & even & $2314$ & 2 & even & 1 & odd \\
$\ov{2}3\ov{1}4$ & 2 & even & 2 & even & $\ov{1}3\ov{2}4$ & 3 & odd & 0 & odd\\
$\ov{3}4\ov{1}2$ & (2,2) & even & 2 & even & $\ov{3}4\ov{2}1$ & (3,2) & odd & 2 & odd\\
$\ov{1}4\ov{3}2$ & (4,1) & odd & 0 & odd & $\ov{2}4\ov{3}1$ & (4,2) & odd & 2 & odd\\
$24\ov{3}\ov{1}$ & (4,2) & odd & 1 & even & $14\ov{3}\ov{2}$ & (4,3) & even & 0 & even \\
\hline
\hline
$v$ & $\mu$ & $\ell_k(\mu)$ & type$(\mu)$ & $\epsilon(\mu)$ 
& $s_{\Box}v$ & $\la$ & $\ell_k(\la)$ & type$(\la)$ & $\epsilon(\la)$ \\ \hline
$1324$ & 1 & even & 0 & even & $\ov{2}3\ov{1}4$ & 2 & even & 2 & even \\
$2314$ & 2 & even & 1 & odd & $\ov{1}3\ov{2}4$ & 3 & odd & 0 & odd \\
$3412$ & (2,2) & even & 1 & odd & $34\ov{2}\ov{1}$ & (3,2) & odd & 1 & even\\
$\ov{1}4\ov{3}2$ & (4,1) & odd & 0 & odd & $24\ov{3}\ov{1}$ & (4,2) & odd & 1 & even\\
$\ov{2}4\ov{3}1$ & (4,2) & odd & 2 & odd & $14\ov{3}\ov{2}$ & (4,3) & even & 0 & even \\
\hline 
\end{tabular}
\label{mula}
\end{table}}}

Let {\bf R} denote the ordered alphabet 
\[
(n-1)'<\cdots 2'<1'< 1,1^\circ < 2,2^\circ < 3,3^\circ < \cdots <1''<2''<\cdots<(n-1)''. 
\]
The single and double primed symbols in {\bf R} are said to be {\em
  marked}, while the rest are {\em unmarked}.  A {\em typed
  $k'$-tableau} $T$ of shape $\la/\mu$ is a sequence of typed
$k$-strict partitions
\[
\mu = \la^0\subset\la^1\subset\cdots\subset\la^p=\la
\]
such that $\la^i/\la^{i-1}$ is a typed $k'$-horizontal strip for
$1\leq i\leq p$. We represent $T$ by a filling of the boxes in
$\la/\mu$ with unmarked symbols of {\bf R} such that for each $i$, the
boxes in $T$ with entry $i$ or $i^\circ$ form the skew diagram
$\la^i/\la^{i-1}$, and we use $i$ (respectively, $i^\circ$) if and only if
$\type(\la^i)\neq 2$ (respectively, $\type(\la^i)=2$), for each $i\in [1,p]$.
For any typed $k'$-tableau $T$ we define $n(T):=\sum_i
n'(\la^i/\la^{i-1})$ and set $z^T:=\prod_i z_i^{n_i}$, 
where $n_i$ denotes the number of times that $i$ or $i^\circ$ appears in $T$.
According to \cite[Thm.\ 4]{T3}, the type D Stanley function $E_w(Z)$
satisfies the equation
\begin{equation}
\label{Etabeq}
E_w(Z) = \sum_T 2^{n(T)} z^T
\end{equation}
summed over all typed $k'$-tableaux $T$ of shape $\la/\mu$.

\begin{defn}
\label{Dtrit}
A {\em typed $k'$-tritableau} $U$ of shape $\la/\mu$ is a filling of
the boxes in $\la/\mu$ with elements of {\bf R} which is weakly
increasing along each row and down each column, such that (i) for each
$a$ in {\bf R}, the boxes in $\la/\mu$ with entry $a$ form a typed
$k'$-horizontal strip, which is a typed $x$-strip (respectively, typed
$y$-strip) if $a\in [(n-1)', 1']$ (respectively, $a\in [1'',(n-1)'']$) and
non-extremal if $a\leq 2'$ (respectively, $a\geq 2''$), (ii) the unmarked
entries of $U$ form a typed $k'$-tableau $T$, and (iii) for $1\leq i
\leq \ell_k(\mu)$ (respectively, $1\leq i \leq \ell_k(\la)$) and $1\leq j
\leq k$, the entries of $U$ in row $i$ are $\geq (\mu_i-k+1)'$
(respectively, $\leq (\la_i-k)''$) and the entries in column $k+1-j$ lie in
the interval $[|w_\mu(j)|', |w_\la(j)-1|'']$. Let
\[
n(U):=n(T) \quad \mathrm{and} \quad
(xyz)^U:= z^T\prod_ix_i^{n'_i}\prod_i (-y_i)^{n''_i} 
\]
where $n'_i$ and $n''_i$ denote the number of times that $i'$ and $i''$
appear in $U$, respectively.
\end{defn}

\begin{thm}
\label{Dskew} For the skew element $w:=w_{\la}w_{\mu}^{-1}$, we have
\begin{equation}
\label{DSteq}
\DS_w(Z;X,Y)=\sum_U 2^{n(U)}(xyz)^U
\end{equation}
summed over all typed $k'$-tritableaux $U$ of shape $\la/\mu$.
\end{thm}
\begin{proof}
We deduce from formula (\ref{dbleD}) that
\begin{equation}
\label{DSw}
\DS_w(Z;X,Y) 
= \sum (-y_{n-1})^{\ell(v_{n-1})}\cdots (-y_1)^{\ell(v_1)}
E_\tau(Z) x_1^{\ell(u_1)}\cdots x_{n-1}^{\ell(u_{n-1})}
\end{equation}
where the sum is over all reduced factorizations $v_{n-1}\cdots v_1 \tau
u_1\cdots u_{n-1}$ of $w$ such that $v_p\in S_n$ is increasing up from $p$ and $u_p\in S_n$
is decreasing down to $p$ for each $p$.
Such factorizations correspond to sequences of typed $k$-strict partitions
\[
\mu = \la^{0} \subset \la^1 \subset \cdots \subset \la^{n-1}
\subset \la^n \subset \la^{n+1} \subset \cdots \subset \la^{2n-1} =\la
\]
with $\la^i/\la^{i-1}$ a typed $x$-strip for each $i\leq n-1$, a typed
$y$-strip for each $i\geq n+1$, and $(\la^n,\la^{n-1})$ a compatible
pair.  We extend each typed $k'$-tableau $T$ on $\la^n/\la^{n-1}$ to a
filling $U$ of the boxes in $\la/\mu$ by placing the entry $(n-i)'$ in
each box of $\la^i/\la^{i-1}$ for $1\leq i\leq n-1$ and $(i-n)''$ in
each box of $\la^i/\la^{i-1}$ for $n+1\leq i\leq 2n-1$. As in the
proof of Theorem \ref{Cskew}, referring this time to Section
\ref{Getks}, one checks that the marked entries of $U$ are restricted
in accordance with Definition \ref{Dtrit}. We deduce that every such
filling $U$ of $\la/\mu$ is a typed $k'$-tritableau on $\la/\mu$ such
that $(xyz)^U=(-y_{n-1})^{\ell(v_{n-1})}\cdots
(-y_1)^{\ell(v_1)}z^Tx_1^{\ell(u_1)}\cdots
x_{n-1}^{\ell(u_{n-1})}$. Conversely, the typed $k'$-tritableaux of
shape $\la/\mu$ correspond to reduced factorizations of $w$ of the
required form.  Finally, by combining (\ref{DSw}) with (\ref{Etabeq}),
we obtain (\ref{DSteq}).
\end{proof}

\begin{example} 
Following \cite{T3}, for any typed $k$-strict partition $\la$, there is a double eta
polynomial $\Eta_\la(c\, |\, t)$, whose image $\Eta_\la(Z; X,Y)$ in
the ring of type D Schubert polynomials is equal to the Grassmannian Schubert
polynomial $\DS_{w_\la}(Z; X,Y)$. Formula (\ref{DSteq}) therefore
gives
\begin{equation}
\label{Eteq}
\Eta_\la(Z;X,Y)=\sum_U 2^{n(U)}(xyz)^U
\end{equation}
summed over all typed $k'$-tritableaux $U$ of shape $\la$. Equation
(\ref{Eteq}) extends \cite[Thm.\ 3]{T3} from single to double eta
polynomials.
\end{example}

\begin{example} 
We extend \cite[Example 2]{T3} to include typed $k'$-tritableaux. Let
$k:=1$, $\la:=(3,1)$ of type $1$, and $Z_2:=(z_1,z_2)$. We have
$w_\la=(2,\ov{3},\ov{1})$ and will compute $\Eta_{(3,1)}(Z_2;X,Y)=
\DS_{2\ov{3}\ov{1}}(Z_2;X,Y)$. Consider the alphabet $\text{\bf
  R}_{1,2}=\{2'<1'<1,1^\circ < 2, 2^\circ <1''<2''\}$. The thirteen
typed $1'$-tritableaux of shape $\la$ with entries in
$\{2',1',1,1^\circ, 2, 2^\circ\}$ are listed in loc.\ cit.  There are
fourteen further $1'$-tritableaux of shape $\la$ involved. The two
tritableaux $\dis \begin{array}{l} 1'\,2\,1'' \\ 1 \end{array}$ and
$\dis \begin{array}{l} 1'\,2\,2'' \\ 1 \end{array}$ satisfy $n(U)=1$,
while the twelve tritableaux
\begin{gather*}
\begin{array}{l} 1\,2\,1'' \\ 1 \end{array}  \ \ \
\begin{array}{l} 1\,2\,2'' \\ 1 \end{array}  \ \ \
\begin{array}{l} 1\,2\,1'' \\ 2 \end{array}  \ \ \
\begin{array}{l} 1\,2\,2'' \\ 2 \end{array}  \ \ \
\begin{array}{l} 1'\,1\,1'' \\ 1 \end{array}  \ \ \
\begin{array}{l} 1'\,1\,2'' \\ 1 \end{array}  \\
\begin{array}{l} 1'\,2\,1'' \\ 2 \end{array}  \ \ \
\begin{array}{l} 1'\,2\,2'' \\ 2 \end{array}  \ \ \
\begin{array}{l} 1'\,1\,1'' \\ 1' \end{array}  \ \ \
\begin{array}{l} 1'\,1\,2'' \\ 1' \end{array}  \ \ \
\begin{array}{l} 1'\,2\,1'' \\ 1' \end{array}  \ \ \
\begin{array}{l} 1'\,2\,2'' \\ 1' \end{array}  
\end{gather*}
satisfy $n(U)=0$.  We deduce from \cite[Example 2]{T3} and Theorem
\ref{Dskew} that
\begin{align*}
\Eta_{3,1}(Z_2; X, Y) 
&= \Eta_{3,1}(Z_2; X) -(y_1+y_2)(z_1^2z_2+z_1z_2^2) \\
&\qquad - (y_1+y_2)(z_1^2+2z_1z_2+z_2^2)x_1 -(y_1+y_2)(z_1+z_2)x_1^2 \\
&= \Eta_{3,1}(Z_2; X) - (y_1+y_2)\Eta_{2,1}(Z_2; X).
\end{align*}
For $\la:=(3,1)$ with $\type(\la)=2$, we have
$w_\la=(\ov{2},\ov{3},1)$ and will compute $\Eta'_{(3,1)}(Z_2;X,Y)=
\DS_{\ov{2}\ov{3}1}(Z_2;X,Y)$. Here, as in op.\ cit., the prime in
$\Eta'_{(3,1)}$ indicates that the indexing partition has type 2.  It
is shown in \cite{T3} that in this case there are six typed
$1'$-tritableaux of shape $\la$ with entries in $\{2',1',1,1^\circ, 2,
2^\circ\}$.  There are 23 further $1'$-tritableaux of the same shape
$\la$. The three tritableaux
\[
\begin{array}{l} 1\,1\,2 \\ 1'' \end{array}  \ \ \ \
\begin{array}{l} 1'\,1\,2 \\ 1'' \end{array}  \ \ \ \
\begin{array}{l} 1^\circ\,1''\,2'' \\ 2^\circ \end{array}
\]
satisfy $n(U)=1$, while the twenty tritableaux 
\begin{gather*}
\begin{array}{l} 1\,1\,1 \\ 1'' \end{array}  \ \ \ 
\begin{array}{l} 2\,2\,2 \\ 1'' \end{array}  \ \ \
\begin{array}{l} 1\,2\,2 \\ 1'' \end{array}  \ \ \
\begin{array}{l} 1^\circ\,2\,2 \\ 1'' \end{array} \ \ \
\begin{array}{l} 1'\,1\,1 \\ 1'' \end{array}  \ \ \
\begin{array}{l} 1'\,2\,2 \\ 1'' \end{array}  \ \ \ 
\begin{array}{l} 1\,1\,1'' \\ 1'' \end{array} \ \ \
\begin{array}{l} 2\,2\,1'' \\ 1'' \end{array}  \\
\begin{array}{l} 1\,2\,1'' \\ 1'' \end{array}  \ \ \ \
\begin{array}{l} 1^\circ\,2\,1'' \\ 1'' \end{array} \ \ \ \
\begin{array}{l} 1'\,1\,1'' \\ 1'' \end{array}  \ \ \ \
\begin{array}{l} 1'\,2\,1'' \\ 1'' \end{array}  \ \ \ \
\begin{array}{l} 1^\circ\,2^\circ\,1'' \\ 1^\circ \end{array}  \ \ \ \
\begin{array}{l} 1^\circ\,2^\circ\,2'' \\ 1^\circ \end{array}  \\
\begin{array}{l} 1^\circ\,2^\circ\,1'' \\ 2^\circ \end{array}  \ \ \
\begin{array}{l} 1^\circ\,2^\circ\,2'' \\ 2^\circ \end{array}  \ \ \
\begin{array}{l} 1^\circ\,1''\,2'' \\ 1^\circ \end{array}  \ \ \
\begin{array}{l} 2^\circ\,1''\,2'' \\ 2^\circ \end{array}  \ \ \
\begin{array}{l} 1^\circ\,1''\,2'' \\ 1'' \end{array}  \ \ \
\begin{array}{l} 2^\circ\,1''\,2'' \\ 1'' \end{array}  
\end{gather*}
satisfy $n(U)=0$. We deduce from \cite[Example 2]{T3} and Theorem
\ref{Dskew} that
\begin{align*}
\Eta'_{3,1}(Z_2; X, Y) 
&= \Eta'_{3,1}(Z_2; X) -y_1(z_1^3 + 2z_1^2z_2+2z_1z_2^2+z_2^3+(z_1^2+2z_1z_2+z_2^2)x_1) \\
&\qquad - (y_1+y_2)(z_1^2z_2+z_1z_2^2)+y_1^2(z^2_1+2z_1z_2+z^2_2+(z_1+z_2)x_1)  \\
&\qquad + y_1y_2(z_1^2+2z_1z_2+z_2^2)-y_1^2y_2(z_1+z_2)  \\
&= \Eta'_{3,1}(Z_2; X) -y_1\Eta_3(Z_2;X)-(y_1+y_2)\Eta'_{2,1}(Z_2; X) \\
&\qquad + y_1^2\Eta_2(Z_2;X) + y_1y_2\Eta'_{1,1}(Z_2;X) -y_1^2y_2\Eta_1'(Z_2;X).
\end{align*}
The $y$-factors in the last equality are exactly the type A single Schubert polynomials
$\AS_\om(-Y)$ for $\om\in S_3$. Since $w_{\la} = s_1s_2s_1s_\Box$, this is in agreement with
\cite[Cor.\ 1]{T5}.
\end{example}

\begin{example}
Let $w:=w_{\la}w_{\mu}^{-1}$ be a skew element of $W_\infty$ or
$\wt{W}_\infty$. Extend the alphabets {\bf Q} and {\bf R} to include
all primed and double primed positive integers, omit the bounds on the
entries of the tritableaux found in Definitions \ref{Ctrit} and
\ref{Dtrit}, and the non-extremal condition in the latter. Then the
right hand sides of equations (\ref{CSteq}) and (\ref{DSteq}) give
tableau formulas for the type C {\em double mixed Stanley function}
$J_w(Z;X,Y)$ of \cite[Ex.\ 3]{T2} and its type D analogue
$I_w(Z;X,Y)$, respectively.
\end{example}

\end{document}